\documentclass[a4paper,12pt,reqno]{amsart}
\usepackage{amsmath,amsthm,amssymb}
\usepackage{graphicx, color}
\usepackage[numbers,sort&compress]{natbib}
\nonstopmode \numberwithin{equation}{section}

\newtheorem{theorem}{Theorem}[section]

 \newtheorem{corollary}{Corollary}[section]

\allowdisplaybreaks

\allowdisplaybreaks

\begin{document}
\title[ Some integrals involving generalized $\mathtt{k}$-Struve functions]
{Some integrals involving generalized $\mathtt{k}$-Struve functions.}

\author{Kottakkaran Sooppy  Nisar}
\address{Department of Mathematics, College of Arts and Science-Wadi Al dawser, 11991,\\
Prince Sattam bin Abdulaziz University, Saudi Arabia}
\email{ksnisar1@gmail.com, n.sooppy@psau.edu.sa}

\author{Saiful Rahman Mondal}
\address{Department of Mathematics\\
King Faisal University, Al Ahsa 31982, Saudi Arabia}
\email{smondal@kfu.edu.sa}


\begin{abstract}
The close form of some integrals involving recently developed generalized $\mathtt{k}$-Struve functions is obtained. The outcome of these  integrations  is  expressed in terms of generalized Wright functions. Several special cases are deduced which lead to some known results. 
\end{abstract}

\maketitle
\section{Introduction and preliminaries}\label{Intro}
The Struve function of order $p$ given by%
\begin{equation}
\mathtt{H}_{p}\left( x\right) :=\sum_{k=0}^{\infty }\frac{\left( -1\right) ^{k}}{%
\Gamma \left( k+3/2\right) \Gamma \left( k+p+\frac{3}{2}\right) }\left( 
\tfrac{x}{2}\right) ^{2k+p+1},  \label{Struve-1}
\end{equation}
is a particular solution of the non-homogeneous Bessel differential equation 
\begin{equation}
x^{2}y^{^{\prime \prime }}\left( x\right) +xy^{^{\prime }}\left( x\right)
+\left( x^{2}-p^{2}\right) y\left( x\right) =\frac{4\left( \frac{x}{2}%
\right) ^{p+1}}{\sqrt{\pi }\Gamma \left( p+1/2\right) }.  \label{bde-1}
\end{equation}
Recently, Nisar et. al,\cite{Nisar-Saiful} introduced and studied various properties 
of $\mathtt{k}$-Struve function $\mathtt{S}_{\nu,c}^{\mathtt{k}}$ defined by
\begin{equation}\label{k-Struve}
\mathtt{S}_{\nu,c}^{\mathtt{k}}(x):=\sum_{r=0}^{\infty}\frac{(-c)^r}
{\Gamma_{\mathtt{k}}(r\mathtt{k}+\nu+\frac{3\mathtt{k}}{2})\Gamma(r+\frac{3}{2})}
\left(\frac{x}{2}\right)^{2r+\frac{\nu}{\mathtt{k}}+1}.
\end{equation}

The generalized Wright hypergeometric function ${}_{p}\psi _{q}(z)$ is given by
the series 
\begin{equation}
{}_{p}\psi _{q}(z)={}_{p}\psi _{q}\left[ 
\begin{array}{c}
(a_{i},\alpha _{i})_{1,p} \\ 
(b_{j},\beta _{j})_{1,q}%
\end{array}%
\bigg|z\right] =\displaystyle\sum_{k=0}^{\infty }\dfrac{\prod_{i=1}^{p}%
\Gamma (a_{i}+\alpha _{i}k)}{\prod_{j=1}^{q}\Gamma (b_{j}+\beta _{j}k)}%
\dfrac{z^{k}}{k!},  \label{Fox-Wright}
\end{equation}%
where $a_{i},b_{j}\in \mathbb{C}$, and real $\alpha _{i},\beta _{j}\in 
\mathbb{R}$ ($i=1,2,\ldots ,p;j=1,2,\ldots ,q$). Asymptotic behavior of this
function for large values of argument of $z\in {\mathbb{C}}$ were studied in 
\cite{CFox} and under the condition 
\begin{equation}
\displaystyle\sum_{j=1}^{q}\beta _{j}-\displaystyle\sum_{i=1}^{p}\alpha
_{i}>-1  \label{eqn-5-Struve}
\end{equation}%
was found in the work of \cite{Wright-2,Wright-3}. Properties of this
generalized Wright function were investigated in \cite{Kilbas}, (see also 
\cite{Kilbas-itsf, Kilbas-frac}. In particular, it was proved \cite{Kilbas}
that ${}_{p}\psi _{q}(z)$, $z\in {\mathbb{C}}$ is an entire function under
the condition ($\ref{eqn-5-Struve}$).

The $k$-gamma function defined in \cite{Diaz2007} as:
 \begin{eqnarray}
 \Gamma_k(z)=\int\limits_{0}^{\infty}t^{z-1}e^{-\frac{t^k}{k}}dt, z\in\mathbb{C}.
 \end{eqnarray}
 By inspection the following relation holds:
 \begin{eqnarray}
 \Gamma_k(z+k)=z\Gamma_k(z)
 \end{eqnarray}
 and
 \begin{eqnarray}\label{2}
 \Gamma_k(z)=k^{\frac{z}{k}-1}\Gamma(\frac{z}{k}).
 \end{eqnarray}
If $k\rightarrow 1$ and $c=1$, then the generalized $k$-Bessel function defined in (\ref{9}) reduces to the well known classical Bessel function $J_v$ defined in \cite{Erdelyi1953}.
 
In this paper, we define a class of integral formulas which involves the generalized $\mathtt{k}$-Struve function as defined in \eqref{k-Struve}. Also, we investigate some special cases as the corollaries. For this purpose, we recall the following result of Lavoie and Trottier \cite{Lavoie1969}.
\begin{align}\label{Lavoi-formula}
&\int\limits_{0}^{1}x^{\alpha-1}\left(1-x\right)^{2\beta-1}\left(1-\frac{x}{3}\right)^{2\alpha-1}\left(1-\frac{x}{4}\right)^{\beta-1}dx\notag\\
&=\left(\frac{2}{3}\right)^{2\alpha}\frac{\Gamma(\alpha)\Gamma(\beta)}{\Gamma(\alpha+\beta)},\mathfrak{R}(\alpha)>0,\mathfrak{R}(\beta)>0.
\end{align}
For more details about integral representations of various special function, the reader can refer these papers (see \cite{Choi2013}, \cite{Choi2014},  \cite{Menaria2016},  \cite{Nisar}, \cite{Nisara}).
 %%%%%%%%%%%%%%%%%%%%%%%%%%%%%%%%%%%%%%%%%%%%%%%%%%%%%%%%%%%%%%%%%%%%%%
 \section{Main Result}
 In this section, we establish two generalized integral formulas containing \eqref{k-Struve}, which represented in terms of generalized Wright function defined in \eqref{Fox-Wright} by inserting with the suitable argument defined in \eqref{Lavoi-formula}.
 \begin{theorem}\label{2.1}
 For  $\mathtt{k}\in \mathbb{R}, \alpha, \mu,\beta,\nu, c\in\mathbb{C}$ with $\nu>\frac{3}{2}\mathtt{k}$, $\Re(\alpha+\mu)>0$, $\Re(\alpha+\frac{\nu}{\mathtt{k}}+1)>0$ and $x>0$, then the following integral formula hold true:
 \begin{align*}
 \int\limits_{0}^{1}x^{\alpha+\mu-1}(1-x)^{2\alpha-1}(1-\frac{x}{3})^{2(\alpha+\mu)-1}(1-\frac{x}{4})^{\alpha-1}
 \mathtt{S}_{\nu,c}^{\mathtt{k}}\left(\frac{y\left(1-\frac{x}{4}\right)\left( 1-x\right)^2}{2}\right)dx
 \end{align*}
 \begin{align}\label{9}
 &=\frac{(\frac{y}{2})^{\frac{\nu}{\mathtt{k}}+1}\Gamma(\alpha+\mu)(\frac{2}{3})^{2(\alpha+\mu)}}{k^\frac{\nu}{k}}\notag\\
 &\times \quad_2\Psi_{3}
\left[
\begin{array}{ccc}
 (\alpha+\frac{\nu}{\mathtt{k}}+1, 2), (1, 1); \\
 &   & \quad | -\frac{cy^2}{4\mathtt{k}}\\
(\frac{\nu}{\mathtt{k}}+\frac{3}{2},1), (\frac{3}{2}, 1), (2\alpha+\frac{\nu}{\mathtt{k}}+\mu, 2) \\
\end{array}
\right].
 \end{align}
 \end{theorem}

 \begin{proof}
 Let $I$ be the left hand side of \eqref{2.1} and applying \eqref{k-Struve} to the integrand of \eqref{9}, we have
 \begin{align*}
&\int\limits_{0}^{1}x^{\alpha+\mu-1}(1-x)^{2\alpha-1}(1-\frac{x}{3})^{2(\alpha+\mu)-1}(1-\frac{x}{4})^{\alpha-1}
\mathtt{S}_{\nu,c}^{\mathtt{k}}\left(\frac{y\left(1-\frac{x}{4}\right)\left( 1-x\right)^2}{2}\right)dx\\
&= \int\limits_{0}^{1}x^{\alpha+\mu-1}(1-x)^{2\alpha-1}(1-\frac{x}{3})^{2(\alpha+\mu)-1}(1-\frac{x}{4})^{\alpha-1}\\
&\times\sum\limits_{r=0}^{\infty}\frac{(-c)^r(\frac{y}{2})^{2r+\frac{\nu}{\mathtt{k}+1}}\left(1-\frac{x}{4}\right)^{2r+\frac{\nu}{\mathtt{k}}}\left(1-x\right)^{2(2r+\frac{\nu}{\mathtt{k}+1})}}{\Gamma_k(r\mathtt{k}+\nu+\frac{3}{2}\mathtt{k})\Gamma{(r+\frac{3}{2})}}dx
 \end{align*}
 Interchanging the order of integration and summation, which is verified by the uniform convergence of the series under the given assumption of theorem \ref{2.1}, we have
 \begin{align*}
I&=\sum\limits_{r=0}^{\infty}\frac{(-c)^r(\frac{y}{2})^{2r+\frac{\nu}{\mathtt{k}}+1}}{\Gamma_k(r\mathtt{k}+\nu+\frac{3}{2}\mathtt{k})\Gamma{(r+\frac{3}{2})}}\\
&\times\int\limits_{0}^{1}x^{\alpha+\mu-1}(1-x)^{2(2r+\alpha+\frac{\nu}{\mathtt{k}}+1)-1)}(1-\frac{x}{3})^{2(\alpha+\mu)-1}(1-\frac{x}{4})^{(\alpha+\frac{\nu}{\mathtt{k}}+1+2r)-1}dx,
\end{align*}
From the assumption given in theorem \ref{2.1}, since $\Re(\frac{\nu}{\mathtt{k}})>0$, $\Re(\alpha+\frac{\nu}{\mathtt{k}}+1+2r)>\Re(\alpha+\frac{\nu}{\mathtt{k}}+1)>0$, $\Re(\alpha+\mu)>0$, $\mathtt{k}>0$ and using \eqref{Lavoi-formula}, we get
 \begin{align*}
I&=\sum\limits_{r=0}^{\infty}\frac{(-c)^r(\frac{y}{2})^{2r+\frac{\nu}{\mathtt{k}}+1}}{\Gamma_k(r\mathtt{k}+\nu+\frac{3}{2}\mathtt{k})\Gamma{(r+\frac{3}{2})}}\left(\frac{2}{3}\right)^{2(\alpha+\mu)}\frac{\Gamma(\alpha+\mu)\Gamma(2r+\alpha+\frac{\nu}{\mathtt{k}}+1)}{\Gamma(2r+2\alpha+\mu+\frac{\nu}{\mathtt{k}}+1)}
\end{align*}
 Using (\ref{2}), we get
 \begin{align*}
I&=\left(\frac{2}{3}\right)^{2(\alpha+\mu)}\frac{\Gamma(\alpha+\mu)\left(\frac{y}{2}\right)^{\frac{\nu}{\mathtt{k}}}}{\mathtt{k}^{\frac{\nu}{\mathtt{k}}+\frac{1}{2}}}\frac{\Gamma(2r+\alpha+\frac{\nu}{\mathtt{k}}+1)}{\Gamma(2r+2\alpha+\mu+\frac{\nu}{\mathtt{k}}+1)\Gamma_k(r\mathtt{k}+\nu+\frac{3}{2}\mathtt{k})\Gamma{(r+\frac{3}{2})}}
\end{align*}
 In view of definition of Fox-Wright function \eqref{Fox-Wright}, we get the required result.
 \end{proof}

 \begin{theorem}\label{2.2}
 For  $\mathtt{k}\in \mathbb{R}, \alpha, \mu,\beta,\nu, c\in\mathbb{C}$ with $\nu>\frac{3}{2}\mathtt{k}$, $\Re(\alpha+\mu)>0$, $\Re(\alpha+\frac{\nu}{\mathtt{k}}+1)>0$ and $x>0$, then the following integral formula hold true:
 \begin{align*}
 \int\limits_{0}^{1}x^{\alpha-1}(1-x)^{2(\alpha+\mu)-1}(1-\frac{x}{3})^{2\alpha-1}(1-\frac{x}{4})^{(\alpha+\mu)-1}
 \mathtt{S}_{\nu,c}^{\mathtt{k}}\left(\frac{yx\left(1-\frac{x}{3}\right)^{2}}{2}\right)dx
 \end{align*}
 \begin{align}\label{10}
 &=\frac{(\frac{y}{2})^{\frac{\nu}{\mathtt{k}}+1}\Gamma(\alpha+\mu)(\frac{2}{3})^{2\alpha}}{k^\frac{\nu}{\mathtt{k}}3^{\frac{2\nu}{\mathtt{k}}+2}}\notag\\
 &\times \quad_2\Psi_{3}
\left[
\begin{array}{ccc}
 (\alpha+\frac{\nu}{\mathtt{k}}+1, 2), (1, 1); \\
 &   & \quad | -\frac{cy^2}{4\mathtt{k}}\\
(\frac{\nu}{\mathtt{k}}+\frac{3}{2},1), (\frac{3}{2}, 1), (2\alpha+\frac{\nu}{\mathtt{k}}+\mu, 2) \\
\end{array}
\right].
 \end{align}
 \end{theorem}

 \begin{proof}
Let $\ell$ be the left hand side of \eqref{2.2} and applying \eqref{k-Struve} to the integrand of \eqref{9}, we have
 \begin{align*}
 &\int\limits_{0}^{1}x^{\alpha-1}(1-x)^{2(\alpha+\mu)-1}(1-\frac{x}{3})^{2\alpha-1}(1-\frac{x}{4})^{(\alpha+\mu)-1}
 \mathtt{S}_{\nu,c}^{\mathtt{k}}\left(\frac{yx\left(1-\frac{x}{3}\right)^{2}}{2}\right)dx\\
&= \int\limits_{0}^{1}x^{\alpha-1}(1-x)^{2(\alpha+\mu)-1}(1-\frac{x}{3})^{2\alpha-1}(1-\frac{x}{4})^{(\alpha+\mu)-1}\\
&\times\sum\limits_{r=0}^{\infty}\frac{(-c)^r(\frac{xy}{2})^{2r+\frac{\nu}{\mathtt{k}+1}}\left(1-\frac{x}{3}\right)^{2(2r+\frac{\nu}{\mathtt{k}}+1}}{\Gamma_k(r\mathtt{k}+\nu+\frac{3}{2}\mathtt{k})\Gamma{(r+\frac{3}{2})}}dx
 \end{align*}
 Interchanging the order of integration and summation, which is verified by the uniform convergence of the series under the given assumption of theorem \ref{2.2}, we have
 \begin{align*}
\ell&=\sum\limits_{r=0}^{\infty}\frac{(-c)^r(\frac{y}{2})^{2r+\frac{\nu}{\mathtt{k}}+1}}{\Gamma_k(r\mathtt{k}+\nu+\frac{3}{2}\mathtt{k})\Gamma{(r+\frac{3}{2})}}\\
&\times\int\limits_{0}^{1}x^{(\alpha+\frac{\nu}{\mathtt{k}}+1+2r)-1}(1-x)^{2(\alpha+\mu)-1}(1-\frac{x}{3})^{2(\alpha+\mathtt{k}+1+2r)-1}(1-\frac{x}{4})^{(\alpha+\mu)-1}dx,
\end{align*}
Since $\Re(\frac{\nu}{\mathtt{k}})>0$, $\Re(\alpha+\frac{\nu}{\mathtt{k}}+1+2r)>\Re(\alpha+\frac{\nu}{\mathtt{k}}+1)>0$, $\Re(\alpha+\mu)>0$, $\mathtt{k}>0$ and using \eqref{Lavoi-formula}, we get
 \begin{align*}
\ell&=\sum\limits_{r=0}^{\infty}\frac{(-c)^r(\frac{y}{2})^{2r+\frac{\nu}{\mathtt{k}}+1}}{\Gamma_k(r\mathtt{k}+\nu+\frac{3}{2}\mathtt{k})\Gamma{(r+\frac{3}{2})}}\left(\frac{2}{3}\right)^{2(\alpha+\frac{\nu}{\mathtt{k}}+1+2r)}\frac{\Gamma(\alpha+\mu)\Gamma(2r+\alpha+\frac{\nu}{\mathtt{k}}+1)}{\Gamma(2r+2\alpha+\mu+\frac{\nu}{\mathtt{k}}+1)}
\end{align*}
 Using (\ref{2}), we obtain
 \begin{align*}
\ell&=\left(\frac{2}{3}\right)^{2(\alpha+\frac{\nu}{\mathtt{k}}+1)}\frac{\Gamma(\alpha+\mu)\left(\frac{y}{2}\right)^{\frac{\nu}{\mathtt{k}}+1}}{\mathtt{k}^{\frac{\nu}{\mathtt{k}}+\frac{1}{2}}}\frac{\Gamma(2r+\alpha+\frac{\nu}{\mathtt{k}}+1)}{\Gamma(2r+2\alpha+\mu+\frac{\nu}{\mathtt{k}}+1)\Gamma_k(r\mathtt{k}+\nu+\frac{3}{2}\mathtt{k})\Gamma{(r+\frac{3}{2})}}
\end{align*}
 In view of definition of Fox-Wright function \eqref{Fox-Wright}, we get the desired result.
 \end{proof}

\section{Special Cases}
In this section, we obtain the integral representation of Struve function and modified Struve function.\\
\textbf{Case 1.} If we set $c=\mathtt{k}=1$ in \eqref{k-Struve}, then we get Struve function of order $\nu$ as
\begin{equation}\label{Struve}
\mathtt{H}_{\nu}(x):=\sum_{r=0}^{\infty}\frac{(-1)^r}
{\Gamma(r+\nu+\frac{3}{2})\Gamma(r+\frac{3}{2})}
\left(\frac{x}{2}\right)^{2r+\nu+1}.
\end{equation}
\textbf{Case 2.} If we set $c=-1$ and $\mathtt{k}=1$ in \eqref{k-Struve}, then we get modified Struve function,
\begin{eqnarray}\label{modified-Struve}
\mathtt{L}_{\nu}(x):=\sum_{r=0}^{\infty}\frac{1}
{\Gamma(r+\nu+\frac{3}{2})\Gamma(r+\frac{3}{2})}
\left(\frac{x}{2}\right)^{2r+\nu+1}.
\end{eqnarray}

\begin{corollary}
Suppose that the conditions of Theorem \ref{2.1} are satisfied. Then the following integral formula holds: 
 \begin{align*}
\int\limits_{0}^{1}x^{\alpha+\mu-1}(1-x)^{2\alpha-1}(1-\frac{x}{3})^{2(\alpha+\mu)-1}(1-\frac{x}{4})^{\alpha-1}
 \mathtt{H}_{\nu}\left(\frac{y\left(1-\frac{x}{4}\right)\left( 1-x\right)^2}{2}\right)dx
 \end{align*}
 \begin{align}\label{12}
 &=(\frac{y}{2})^{\nu+1}\Gamma(\alpha+\mu)(\frac{2}{3})^{2(\alpha+\mu)}\notag\\
 &\times \quad_2\Psi_{3}
\left[
\begin{array}{ccc}
 (\alpha+\nu+1, 2), (1, 1); \\
 &   & \quad | -\frac{y^2}{4}\\
(\nu+\frac{3}{2},1), (\frac{3}{2}, 1), (2\alpha+\nu+\mu, 2) \\
\end{array}
\right].
 \end{align}
\end{corollary}

  \begin{corollary}
Assume that the conditions of theorem \ref{2.2} satisfied,  then the following integral formula hold true:
 \begin{align*}
\int\limits_{0}^{1}x^{\alpha-1}(1-x)^{2(\alpha+\mu)-1}(1-\frac{x}{3})^{2\alpha-1}(1-\frac{x}{4})^{(\alpha+\mu)-1}
 \mathtt{L}_{\nu}\left(\frac{yx\left(1-\frac{x}{3}\right)^{2}}{2}\right)dx
 \end{align*}
 \begin{align}\label{13}
 &=\frac{(\frac{y}{2})^{\nu+1}\Gamma(\alpha+\mu)(\frac{2}{3})^{2\alpha}}{3^{2\nu+2}}\notag\\
 &\times \quad_2\Psi_{3}
\left[
\begin{array}{ccc}
 (\alpha+\nu+1, 2), (1, 1); \\
 &   & \quad | \frac{y^2}{4}\\
(\nu+\frac{3}{2},1), (\frac{3}{2}, 1), (2\alpha+\nu+\mu, 2) \\
\end{array}
\right].
 \end{align}
 \end{corollary}
%
%\textbf{Acknowledgements}\\
%The authors would like to express profound gratitude to referees for deeper
%review of this paper and the referee's useful suggestions that led to an improved presentation of the paper.\\

\end{document}